\documentclass[11pt]{article}
\usepackage{amsmath,amssymb,amsthm}
\usepackage{enumerate}
\DeclareMathOperator{\GF}{GF}
\DeclareMathOperator{\Tr}{Tr}
\newcommand{\Z}{{\mathbb Z}}
\newtheorem{theorem}{Theorem}[section]
\newtheorem{proposition}[theorem]{Proposition}
\newtheorem{lemma}[theorem]{Lemma}
\newtheorem{conjecture}[theorem]{Conjecture}
\theoremstyle{remark}
\newtheorem{remark}[theorem]{Remark}
\title{Proof of a Conjecture of Helleseth: Maximal Linear Recursive Sequences of Period $2^{2^n}-1$ Never Have Three-Valued Cross-Correlation}
\author{Daniel J. Katz\footnote{Daniel J. Katz is with the Department of Mathematics, Simon Fraser University, Burnaby, BC V5A 1S6, Canada. ({\tt daniel\_katz\_2@sfu.ca})}}
\date{28 October 2011}
\begin{document}
\maketitle
\section*{Abstract}
We prove a conjecture of Helleseth that claims that for any $n \geq 0$, a pair of binary maximal linear sequences of period $2^{2^n}-1$ can not have a three-valued cross-correlation function.

\section{Introduction}

The binary maximal linear sequences of period $2^m-1$ are the sequences of elements in $\GF(2)$ of the form $\{\Tr(\alpha^{d i + t})\}_{i \in \Z}$ where $\alpha$ is a generator of $\GF(2^m)^*$, $\Tr\colon \GF(2^m) \to \GF(2)$ is the absolute trace, and $d$ and $t$ are integers (or integers modulo $2^m-1$) with $\gcd(d,2^m-1)=1$.
(See the Introduction of \cite{Helleseth}.)
The cross-correlation of any two binary sequences $a=\{a_i\}$ and $b=\{b_i\}$ whose periods are divisors of $2^m-1$ is the function $C_{a,b}(t)=\sum_{i=0}^{2^m-2} (-1)^{a_{i-t}+b_i}$.
In this note, we shall take $a=\{a_i\}=\{\Tr(\alpha^{i})\}$ and $b=\{b_i\}=\{\Tr(\alpha^{d i})\}$, where the {\it decimation} d has $\gcd(d,2^m-1)=1$.
We call decimations with $d \equiv 1, 2, \ldots, 2^{m-1} \pmod{2^m-1}$ {\it trivial decimations} because $\{\Tr(\alpha^{2^k i})\}$ is the same sequence as $\{\Tr(\alpha^i)\}$.
One readily shows that $C_{a,b}(t)$ is the same as
$$C_d(t)=\sum_{x \in \GF(2^m)^*} (-1)^{\Tr(\alpha^{-t} x + x^d)}.$$

For a fixed $d$, we are interested in how many different values $C_d(t)$ takes as $t$ varies over $\Z/(2^m-1)\Z$.
We say that $C_d(t)$ is {\it $v$-valued} to mean that $|\{C_d(t)\colon t\in \Z/(2^m-1)\Z\}|=v$.
Helleseth gave the following criterion for determining whether $C_d(t)$ is two-valued.
\begin{theorem}[Helleseth \cite{Helleseth}, Theorem 3.1(d),(g), Theorem 4.1]\label{HellesethsTheorem}
If $d \equiv 1, 2, \ldots, 2^{m-1} \pmod{2^m}$, then $C_d(t) \in \{-1, 2^m-1\}$ for all $t$.  Otherwise, $C_d(t)$ takes at least three different values.
\end{theorem}
In the same paper, Helleseth conjectured the following.
\begin{conjecture}[Cf.~Helleseth \cite{Helleseth}, Conjecture 5.2]\label{HellesethsConjecture}
If $m$ is a power of $2$, $C_d(t)$ is not three-valued.
\end{conjecture}
In view of Theorem \ref{HellesethsTheorem}, this conjecture says that if $m$ is a power of $2$, then $C_d(t)$ is either two-valued (if $d$ is a trivial decimation) or takes four or more values (if $d$ is nontrivial).
We prove this conjecture in this note.

Feng \cite{Feng} recently proved the following weaker form of Conjecture \ref{HellesethsConjecture}.
\begin{theorem}[Feng \cite{Feng}, Theorem 2]\label{FengsTheorem}
If $m$ is a power of $2$ and $C_d(t)=-1$ for some value of $t$, then $C_d(t)$ cannot be three-valued.
\end{theorem}
We prove Conjecture \ref{HellesethsConjecture} by proving the following.
\begin{theorem}\label{OurTheorem}
If $C_d(t)$ is three-valued, then $C_d(t)=-1$ for at least one value of $t$.
\end{theorem}
This, combined with Theorem \ref{FengsTheorem}, immediately implies Conjecture \ref{HellesethsConjecture}.
\begin{remark}
One should note that our theorem does not assume $m$ is a power of $2$, so it is much more general in scope that what is needed.
In fact, one can prove the same theorem for maximal linear sequences derived from fields $\GF(p^m)$ with $p$ odd: this (and more) is done in \cite{Katz}.
\end{remark}

\section{Proof of Theorem \ref{OurTheorem}}

We shall prefer to work in terms of the {\it Walsh transform}, defined as
$$W_d(a)=\sum_{x \in \GF(2^m)} (-1)^{\Tr(x^d+ax)},$$
and it is straightforward to show that
$$W_d(\alpha^{-t})=1+C_d(t).$$
Thus the values of $W_d$ on $\GF(2^m)^*$ are just the values of $C_d$ shifted by $1$.
So $C_d$ is three-valued if and only if $W_d$ is three valued on $\GF(2^m)^*$.

We need to establish a few well-known facts before proceeding to the proof of Theorem \ref{OurTheorem}.
First, we need a simple result which, in rough terms, states that a sequence cannot be perfectly correlated or anti-correlated to a nontrivial decimation of itself.
\begin{lemma}\label{Magnitude}
If $d\not\equiv 1,2,\ldots,2^{m-1} \pmod{2^m-1}$, then $|W_d(a)| < 2^m$.
\end{lemma}
\begin{proof}
From the definition of $W_d(a)$ as the sum $\sum_{x \in \GF(2^m)} (-1)^{\Tr(x^d+a x)}$
of $2^m$ terms in $\{1,-1\}$, it suffices to prove that the said terms are not all of the same sign.
The $x=0$ term is $1$, and so the only way that all the terms can have the same sign is if 
$$\Tr(x^d+a x)=(x^d+x^{2 d} + \cdots + x^{2^{m-1} d}) + a(x+x^2+\cdots+x^{2^{m-1}})$$
equals $0$ for all $x \in \GF(2^m)$, i.e., if and only if this polynomial is zero modulo $x^{2^m} - x$.
Given our assumption on $d$, all the exponents of $x$ that appear in the polynomial as expressed above are distinct modulo $2^m-1$, so this cannot happen.
\end{proof}

We consider the first few power moments of $W_d$, with the $r$th power moment defined to be
$$P_r = \sum_{a \in \GF(2^m)^*} W_d(a)^r,$$
where we use the convention $0^0=1$ in evaluating $P_0$.
The power moments of $C_d$ have been calculated by Helleseth, whence it is easy to obtain those of $W_d$.

\begin{proposition}[See Helleseth \cite{Helleseth}]\label{Moments}
We have
\begin{enumerate}[(a)]
\item $P_0=2^m-1$,
\item $P_1=2^m$,\label{FirstMoment}
\item $P_2=2^{2 m}$, and\label{SecondMoment}
\item $P_3=2^{2m} |V|$,
\end{enumerate}
where $V$ is the set of roots of $1+x^d+(1+x)^d$ in $\GF(2^m)$.
\end{proposition}
From these one can readily deduce the following, which also appears as calculations in \cite{Feng}.
\begin{proposition}\label{Counts}
Suppose that $W_d(a)$ is three-valued on $\GF(2^m)^*$ with values $A$, $B$, and $C$, and that $W_d(a)=C$ for $N_C$ values of $a \in \GF(2^m)^*$.
Then
$$N_C = \frac{2^{2 m} - 2^m (A+B) + (2^m-1) A B}{(C-A)(C-B)}$$
and
$$2^{2 m} |V| = 2^{2 m} (A+B+C) - 2^m(A B + B C + C A) + (2^m-1) A B C,$$
where $V$ is the set of roots of $1+x^d+(1+x)^d$ in $\GF(2^m)$.
\end{proposition}
\begin{proof}
To get $N_C$, compute $\sum_{a \in \GF(2^m)^*} (W_d(a)-A)(W_d-B)$.
On the one hand, $W_d(a) \in \{A,B,C\}$ implies that the sum is $N_C(C-A)(C-B)$.  
On the other hand, one can also calculate the sum in terms of power moments as $P_2 - (A+B) P_1 + A B P_0$, and then use the values given in Proposition \ref{Moments}.
To get $|V|$, one can employ the same approach, this time with the sum $\sum_{a \in \GF(2^m)^*} (W_d(a)-A)(W_d(a)-B)(W_d(a)-C)$: on the one hand, it is zero, and on the other, it can be expressed in terms of $P_0$, $P_1$, $P_2$, and $P_3$.
\end{proof}
This can be used to prove an interesting result about the $2$-divisibility of the values assumed by $W_d(a)$.
\begin{lemma}\label{Divisibility}
Suppose that $W_d(a)$ takes precisely three values $A$, $B$, and $C$ for $a \in \GF(2^m)^*$.  If all three values are non-zero, then $2^{m+1} \mid A B$.
\end{lemma}
\begin{proof}
From Proposition \ref{Counts} we have
\begin{equation}\label{Vequation}
2^{2 m} |V| = 2^{2 m} (A+B+C) - 2^m(A B + B C + C A) + (2^m-1) A B C,
\end{equation}
where $V$ is the set of roots of $1+x^d+(1+x)^d$ in $\GF(2^m)$.
Suppose that $A,B,C\not=0$; then Lemma \ref{Magnitude} shows that $A,B,C\not\equiv 0 \pmod{2^m}$.
(We clearly have a nontrivial decimation by Theorem \ref{HellesethsTheorem} since $W_d$ is three-valued on $\GF(2^m)^*$, and hence $C_d$ is three-valued.)
Then the term $(2^m-1) A B C$ is divisible by fewer powers of $2$ than the other terms on the right hand side of \eqref{Vequation}, so $2^{2 m} |V|$ and $A B C$ have exactly the same power of $2$ in their respective prime factorizations, and so $2^{2 m} | A B C$.
Since $C\not\equiv 0 \pmod{2^m}$, this means that $2^{m+1} | A B$.
\end{proof}

Now we are ready to prove Theorem \ref{OurTheorem}.
We assume that $C_d$ is three-valued and that none of these values is $-1$ in order to show a contradiction.
Then $W_d(a)$ is three-valued for $a \in \GF(2^m)^*$ with the three nonzero values $A$, $B$, $C$.
Note that Proposition \ref{Moments}\eqref{FirstMoment} shows that
$$\sum_{a \in \GF(2^m)^*} W_d(a)=2^m,$$
so we cannot have $A,B,C < 0$.
Furthermore, by parts \eqref{FirstMoment} and \eqref{SecondMoment} of the same proposition,
$$\left(\sum_{a \in \GF(2^m)^*} W_d\right)^2 = 2^{2 m} = \sum_{a \in \GF(2^m)^*} W_d(a)^2,$$
so we cannot have $A,B,C > 0$.
Then without loss of generality, we may take $A < 0 < B$ and $C$ not between $A$ and $B$.
Then by Proposition \ref{Counts}, the number $N_C$ of $a \in \GF(2^m)^*$ such that $W_d(a)=C$ is 
$$N_C = \frac{2^{2 m} - 2^m (A+B) + (2^m-1) A B}{(C-A)(C-B)}.$$
Since $C$ is not between $A$ and $B$, the denominator is positive, so
$$2^{2 m} - 2^m (A+B) + (2^m-1) A B > 0.$$
We use Lemma \ref{Magnitude} and the fact that $A < 0$ and $B > 0$ to see that
$$2^{2 m} - 2^m (-(2^m-1)+1) + (2^m-1) A B > 0,$$
so that $A B > -2^{m+1}$.
But by Lemma \ref{Divisibility} and the fact that $A < 0 < B$, we have that $A B \leq -2^{m+1}$, which gives the contradiction that completes the proof of Theorem \ref{OurTheorem}.

\section*{Acknowledgements}

The author gives warm thanks to Robert Calderbank and Jonathan Jedwab.
He thanks Robert Calderbank for introducing him to this problem, for stimulating discussions, and for unflagging encouragement.
Jonathan Jedwab has also provided great encouragement, and must also be thanked for drawing the author's attention to \cite{Feng} when it appeared.

\end{document}